\documentclass[leqno,11pt]{amsart}
\usepackage{amsmath,amstext,amssymb,amsopn,amsthm,mathrsfs}
\usepackage{multirow}

\numberwithin{equation}{section}

\allowdisplaybreaks

\newtheorem{theor}{Theorem}[section]
\newtheorem{propo}[theor]{Proposition}
\newtheorem{lemma}[theor]{Lemma}

\newtheorem*{rem*}{Remark}

\newcommand{\N}{\mathbb{N}^d}
\newcommand{\R}{\mathbb{R}^d}

\newcommand{\op}{\star} 
\newcommand{\lag}{\widehat{L}}

\begin{document}

\footnotetext{
\noindent \emph{2010 Mathematics Subject Classification:} Primary 42C10; Secondary 47D03, 47G10.\\
\emph{Key words and phrases:} Laguerre semigroup, contractivity, symmetric diffusion semigroup,
	Bessel semigroup. \\
Research of both authors supported by MNiSW Grant N N201 417839.
}

\title[Contractivity of Laguerre semigroups]
	{On $L^p$-contractivity of Laguerre semigroups}


\author[A. Nowak]{Adam Nowak}
\address{Adam Nowak, \newline
			Instytut Matematyczny,
      Polska Akademia Nauk, \newline
      \'Sniadeckich 8,
      00--956 Warszawa, Poland \newline
			\indent and \newline
			Instytut Matematyki i Informatyki,
      Politechnika Wroc\l{}awska,       \newline
      Wyb{.} Wyspia\'nskiego 27,
      50--370 Wroc\l{}aw, Poland      
      }
\email{adam.nowak@pwr.wroc.pl}

\author[K. Stempak]{Krzysztof Stempak}
\address{Krzysztof Stempak,     \newline
      Instytut Matematyki i Informatyki,
      Politechnika Wroc\l{}awska,       \newline
      Wyb{.} Wyspia\'nskiego 27,
      50--370 Wroc\l{}aw, Poland      }  
\email{krzysztof.stempak@pwr.wroc.pl}

\begin{abstract}
We study several Laguerre semigroups appearing in the literature and find sharp ranges of type
parameters for which these semigroups are contractive on all $L^p$ spaces, $1\le p \le \infty$.
We also answer a similar question for Bessel semigroups, 
which in a sense are closely related to the Laguerre semigroups.
\end{abstract}

\maketitle

\section{Introduction} \label{sec:intro}
Let $(X,\mu)$ be a $\sigma$-finite measure space  and assume that $\mathcal D$ is a
linear subspace  of the space $\mathcal M$  of all measurable functions on $X$, such that it contains all
$L^p=L^p(X,\mu)$ spaces,  $1\le p\le\infty$. If $\mu$ is finite then one can simply take $\mathcal D=L^1$;
otherwise the natural choice is $\mathcal{D} = L^1 + L^{\infty}$.
A \textit{symmetric diffusion semigroup}, see \cite[Chapter III]{St}, is a family of linear operators
$\{T_t\}_{t\ge 0}$, $T_0=\textrm{Id}$,  mapping jointly  $\mathcal D$  into $\mathcal M$,  and satisfying: 
\begin{itemize}
\item[(i)] for each $1\le p\le\infty$, 
every $T_t$ is a contraction on $L^p$ and $\{T_t\}_{t\ge 0}$ is a semigroup there,
\item[(ii)] each $T_t$ is a self-adjoint operator in $L^2$, 
\item[(iii)] 	for each $f\in L^2$, $\lim_{t\to0^+}T_tf=f$ in $L^2$.
\end{itemize}
A symmetric diffusion semigroup $\{T_t\}_{t\ge 0}$ is called \emph{Markovian}, 
if it is positive and conservative, that is satisfies in addition
\begin{itemize}
\item[(iv)] for each $t$, $T_tf\ge0$ if $f\ge0$, 
\item[(v)] $T_t\textbf{1}=\textbf{1}$ for all $t$, 
\end{itemize}
respectively.
Semigroups satisfying (i)-(iv) and 
\begin{itemize}
\item[(v')] $T_t\textbf{1}\le\textbf{1}$ for all $t$
\end{itemize}
replacing (v), are called \emph{submarkovian}.

Symmetric diffusion semigroups commonly emerge from unbounded self-adjoint operators
as their infinitesimal generators. Suppose $L$ is a positive self-adjoint operator on $L^2=L^2(X,\mu)$. 
The semigroup of operators $\{\exp(-tL)\}_{t\ge 0}$, defined on $L^2$ by means of the spectral theorem, 
is said to be an \textit{$L^p$-contractive semigroup}, see \cite[Chapter X.8, p.\,255]{RS}, 
if for each $ 1\le p\le\infty$ and every $t\ge 0$ one has
$$
\|\exp(-tL)f\|_p\le\|f\|_p, \qquad f\in L^2\cap L^p(X,\mu).
$$

In concrete realizations, it frequently happens that $\exp(-tL)$, $t>0$, are integral operators,
\begin{equation} \label{ker}
\exp(-tL)f(x)=\int_X K_t(x,y)f(y)\,d\mu(y), \qquad f\in L^2, \quad x\in X,
\end{equation}
and the right-hand side of \eqref{ker} makes sense
for a larger than $L^2$ class of functions, that usually includes all $L^p$ spaces, $1\le p\le\infty$, 
and defines a family of operators $\{T_t\}_{t>0}$ which are bounded on $L^p(X,\mu)$, $1\le p\le\infty$.  
Thus $T_t$, $t>0$, are the unique linear extensions of $\exp(-tL)$ to $L^p$, $1\le p < \infty$, 
and $\{T_t\}_{t> 0}$ augmented by $T_0=\textrm{Id}$ typically occurs to be a
symmetric diffusion semigroup. $L^p$-contractivity is a decisive property in several instances, 
see for example \cite[Chapter III]{St} and \cite{G}. In particular, it is absolutely 
essential for an application of Stein's celebrated maximal theorem \cite[p.\,73]{St}.

The aim of this paper is a thorough  study of the Laguerre semigroups appearing in the literature
from the $L^p$-contractivity property point of view. The detailed study of mapping properties of  
heat-diffusion semigroups associated with various systems of Laguerre functions 
was initiated in \cite{Ste} in the one-dimensional setting and then continued in \cite{N} 
in the multi-dimensional weighted setting. In this article we investigate semigroups associated 
with four different Laguerre systems in the  multi-dimensional setting. 
Let $L_k^{\alpha}$ denote the Laguerre polynomial of degree $k \in \mathbb{N}=\{0,1,2,\ldots\}$
and order $\alpha > -1$, cf. \cite[p.\,76]{Leb}. Starting with the one-dimensional situation, 
for any given $\alpha>-1$ we consider the following systems on $\mathbb{R}_+=(0,\infty)$:
\begin{itemize}
\item
	normalized Laguerre \emph{polynomial} system, $\{{\lag}_k^{\alpha} : k\in \mathbb{N}\}$,
	$$
		{\lag}_{k}^{\alpha}(x) =\left(\frac{\Gamma (k+1)}{\Gamma(k+\alpha +1)}\right)
		^{1 \slash 2} L_{k}^{\alpha}(x), \qquad x > 0,
	$$
	which is an  orthonormal basis in $L^2(\mathbb R_+,\, x^{\alpha}e^{-x}\, dx)$;
\item \emph{standard} Laguerre function system, $\{\mathcal{L}_{k}^{\alpha} : k\in\mathbb N\}$, 
  $$
		\mathcal{L}_{k}^{\alpha}(x) =
		{\lag}_k^{\alpha}(x) x^{\alpha\slash 2}e^{-x\slash 2}, \qquad x > 0,
	$$	
  which is an  orthonormal basis in $L^2(\mathbb R_+,\, dx)$;
\item Laguerre function system of \emph{Hermite type}, $\{\varphi_{k}^{\alpha} : k\in\mathbb N\}$, 
	$$
		\varphi_{k}^{\alpha}(x)=
		\sqrt{2} {\lag}_k^{\alpha}(x^2) x^{\alpha+1\slash 2}e^{-x^2\slash 2}, \qquad x > 0,
	$$	
	which is  an  orthonormal basis in $L^2(\mathbb R_+,\, dx)$;	
\item Laguerre function system of \emph{convolution type}, $\{\ell_{k}^{\alpha} : k\in\mathbb N\}$, 
	$$
		\ell_{k}^{\alpha}(x)=
		\sqrt{2} {\lag}_k^{\alpha}(x^2) e^{-x^2\slash 2}, \qquad x > 0,
	$$		
	which is  an  orthonormal basis in $L^2(\mathbb R_+,\, x^{2\alpha+1}\,dx)$.
\end{itemize}
The corresponding multi-dimensional systems are then formed simply by taking tensor products. 
Thus for a multi-index $\alpha = (\alpha _1, \ldots , \alpha _d) \in (-1,\infty)^{d}$, 
the system $\{{\lag}_k^{\alpha} : k \in \mathbb{N}^d\}$ is an orthonormal basis in $L^2(\R_+,\,m_{\alpha})$,
the systems $\{\mathcal L_{k}^{\alpha} : k \in \N \}$ and $\{ \varphi _{k}^{\alpha} : k \in \N \}$ are 
orthonormal bases in $L^2(\R_{+})$, and  $\{\ell_{k}^{\alpha} : k \in \N \}$ is an orthonormal basis in 
$L^2(\R_{+},\,\mu_\alpha)$, where 
$$
dm_{\alpha}(x)=\Big(\prod_{i=1}^d x_i^{\alpha_i}e^{-x_i}\Big)dx, \qquad
d\mu_\alpha(x)=\Big(\prod_{i=1}^d x_i^{2\alpha_i+1}\Big)dx.
$$

For these four systems, there are naturally associated differential operators
$L_{\alpha}^{{\lag}}$, $L_\alpha^{\mathcal L}$, $L_\alpha^{\varphi}$ and $L_\alpha^{\ell}$, 
for which $\{{\lag}_k^{\alpha}\}$, $\{\mathcal L_{k}^{\alpha}\}$, $\{ \varphi _{k}^{\alpha}\}$ 
and $\{\ell_{k}^{\alpha} \}$, respectively, are systems of eigenfunctions; see Sections 
\ref{sec:pol}-\ref{sec:lag_conv} for the definitions. 
We investigate the semigroups $\{T_t^{\alpha,\op}\}$,
$$
T_t^{\alpha,\op} = \exp(-t L_{\alpha}^\op), \qquad \op = {{\lag}},\; {\mathcal L},\; {\varphi}, \; {\ell},
$$
generated by 
these Laguerre differential operators (or rather by their natural self-adjoint extensions).  
We also study so-called \emph{modified} Laguerre semigroups
$$
\{\widetilde{T}_t^{\alpha,\op,j}\}, \qquad j=1,\ldots,d, \qquad \op = 
	{{\lag}},\; {\mathcal L},\; {\varphi}, \; {\ell},
$$
which emerge naturally in the theory of conjugacy connected with Laguerre expansions. They
are generated by self-adjoint extensions of suitable modifications of the Laguerre operators in
question, see \cite{NS1} for a general background, and are important tools when studying conjugacy
problems. Various properties of the modified Laguerre semigroups can be found in \cite{No2,NS,NS2}.
All the considered Laguerre semigroups possess integral representations, and the corresponding
integral kernels are known explicitly. Since all the kernels are strictly positive, 
it follows that all the semigroups satisfy Condition (iv). Obviously, they also satisfy (ii) and (iii),
by the very definition. However, determining whether Conditions (i) and (v) or (v') are satisfied or not,
requires a more subtle treatment. In this paper we find the answers in all the cases.

Our main results are summarized in the following table, which provides
sharp ranges of the type parameter $\alpha$ for which the Laguerre semigroups are $L^p$-contractive. 
Noteworthy, these are also optimal ranges of $\alpha$ for the assumptions of Stein's maximal
theorem \cite[p.\,73]{St} to be satisfied.

{{
\begin{table*}[htbp]
\centering
\begin{tabular}{c|c|c|c|c|}
\cline{2-5}
 & $\op={{\lag}}$ 
 & $\op={\mathcal{L}}$ 
 & $\op={\varphi}$ 
 & $\op={\ell}$ \\ 
\cline{1-5} 
   \multicolumn{1}{|c|}{$T^{\alpha,\op}_t$} 
   & $\alpha \in (-1,\infty)^d$ 
   & $\alpha \in [0,\infty)^d$ 
   & $\alpha \in \big(\{-1\slash 2\}\cup[1\slash 2,\infty)\big)^d$ 
   & $\alpha \in (-1,\infty)^d$ \\ 
\cline{1-5}   
 \multicolumn{1}{|c|}{$\widetilde{T}^{\alpha,\op,j}_t$} 
 & $\alpha_j \in [-1\slash 2,\infty)$, 
 & $\alpha_j \in [-1,\infty)$, 
 & $\alpha_j \in \{-3\slash 2\}\cup[-1\slash 2,\infty)$, 
 & $\alpha_j \in [-1\slash 2,\infty)$,     \\ [-5pt]
 \multicolumn{1}{|c|}{$j=1,\ldots,d$}
 & $\alpha_i > -1 
 		, \; i \neq j$ 
 & $\alpha_i \ge 0 
 		, \; i \neq j$ 
 & $\alpha_i \in \{-1\slash 2\}\cup[1\slash 2,\infty), \; i \neq j$ 
 & $\alpha_i > -1 
 		, \; i \neq j$     \\ 
\cline{1-5}
\noalign{\bigskip}
\end{tabular}
	\caption{Optimal ranges of $\alpha$ for $L^p$-contractivity}
	\label{tab:Optimal}
\end{table*}
}}

In the cases of Laguerre polynomial system and Laguerre function system of convolution type the 
situation is perfect, as far as the non-modified semigroups are considered: for any $\alpha \in 
(-1,\infty)^{d}$, $\{T_t^{\alpha,{{\lag}}}\}$ and $\{T_t^{\alpha,{\ell}}\}$
are $L^p$-contractive. In the cases of two other 
systems, to ensure the action of the semigroups on all the corresponding $L^p$ spaces, $1\le p\le\infty$, 
one has to restrict the set of parameters $\alpha$, namely $\alpha \in [0,\infty)^{d}$ for 
$\{T_t^{\alpha,{\mathcal L}}\}$, and $\alpha \in [-1/2,\infty)^{d}$ for $\{T_t^{\alpha,{\varphi}}\}$. 
In the case of standard Laguerre function system, in the restricted range $\alpha  \in [0,\infty)^{d}$ the 
situation is again typical: $\{T_t^{\alpha,{\mathcal L}}\}$ is $L^p$-contractive.
A bit surprisingly, this is not the case of the  Laguerre function system of Hermite type. More precisely, 
$L^p$-contractivity holds for $\alpha  \in (\{-1/2\}\cup[1/2,\infty))^{d}$, but fails to hold if 
at least one of the coordinates of $\alpha$ falls into the interval $(-1/2,1/2)$. 
On the other hand, this lack of contractivity is consistent with the 
restriction on $\alpha$ frequently occurring in various results in the context
of the Laguerre system of Hermite type; see \cite[Theorem 3.3]{NS} as a typical example.
Finally, it is remarkable that the ranges of $\alpha$ for $L^p$-contractivity of the modified 
Laguerre semigroups do not coincide with those for the original Laguerre semigroups. 
This, however, is consistent with certain results in the underlying
conjugacy theory, which were proved after restricting to those $\alpha$, for which both the original and the
modified Laguerre semigroups are contractive; see for instance \cite{No2}. 

The paper is organized as follows. In Sections \ref{sec:pol}-\ref{sec:lag_conv} we analyse 
the Laguerre semigroups associated with the systems
$\{{\lag}_k^{\alpha}\}$, $\{\mathcal{L}_k^{\alpha}\}$, $\{\varphi_k^{\alpha}\}$ and 
$\{\ell_k^{\alpha}\}$, respectively. Section \ref{sec:comm} contains comments on $L^p$-contractivity
of certain Bessel semigroups, which in a sense are closely related to the Laguerre semigroups.

We shall use the following notation. By $\textbf{1}$ we always denote the function identically equal 1 on 
its domain. Given a bounded linear operator $A$ on $L^p(X,\mu)$, $1\le p\le\infty$, we shall write 
$\|A\|_{p\to p}$ for its operator norm.

\section{Laguerre polynomial semigroup} \label{sec:pol}

Recall that the system $\{{\lag}_{k}^{\alpha} : k \in \N \}$ of multi-dimensional normalized 
Laguerre polynomials is an orthonormal basis in $L^2(\R_{+}, m_\alpha)$. 
It consists of eigenfunctions of the differential operator
\begin{equation*}
 L^{{\lag}}_{\alpha}=
- \sum _{i=1}^{d}\left( x_i\frac{\partial^2}{\partial x_i^2}+ (\alpha_i+1-x_i)
\frac{\partial}{\partial x_i}\right);
\end{equation*}
we have $L_\alpha {\lag}_{k}^{\alpha}=|k|{\lag}_{k}^{\alpha}$, where $|k|=k_1+\ldots+k_d$ is the 
length of $k$. The operator $L^{{\lag}}_\alpha$ is formally symmetric and positive in 
$L^2(\R_{+}, m_\alpha)$, and admits a natural self-adjoint extension in $L^2(\R_{+}, m_\alpha)$ 
whose spectral decomposition is given by the Laguerre polynomials 
(we use the same symbol $L^{{\lag}}_\alpha$  to denote this extension).
The corresponding heat semigroup $\{T_t^{\alpha,{{\lag}}}\} = \{\exp(-tL^{{\lag}}_\alpha)\}$
is defined by means of the spectral theorem,  
\begin{equation} \label{spec}
T_t^{\alpha,{{\lag}}}f=\sum_{n=0}^{\infty} e^{-tn} \sum_{|k|=n}
\langle f,{\lag}_k^\alpha \rangle_{m_\alpha} {\lag}_k^\alpha, \qquad f\in L^2(\R_+,\, m_\alpha),
\end{equation}
and it has the integral representation
\begin{equation} \label{int_pol} 
  T_t^{\alpha,{{\lag}}}f(x)
  =\int_{\R_{+}} G_t^{\alpha,{\lag}}(x,y)f(y)\,dm_\alpha (y),  \qquad x\in \R_{+}, 
\end{equation}
where the heat kernel is given by
\begin{equation*}
G^{\alpha,{\lag}}_t(x,y)=\sum_{n=0}^\infty e^{-tn} 
\sum_{|k|=n} {\lag}_k^\alpha(x){\lag}_k^\alpha(y),\qquad x,y\in \R_{+}.
\end{equation*}

This oscillating series can be summed by means of
the Hille-Hardy formula \cite[(4.17.6)]{Leb}, and the result is
\begin{equation*}
G^{\alpha,{\lag}}_t(x,y)= \frac{e^{t(|\alpha|+d)/2}}{(2\sinh (t/2))^{d}}\exp\bigg(
{-\frac{e^{-t\slash 2}}{2\sinh(t\slash 2)}\sum_{i=1}^d (x_i+y_i)}\bigg)
\prod^{d}_{i=1}\,(x_i y_i)^{-\alpha_i/2} I_{\alpha_i}\left(\frac{ \sqrt{x_i y_i}}{\sinh (t/2)}\right).
\end{equation*}
Here $|\alpha|=\alpha_1+\ldots+\alpha_d$ (notice that this quantity may be negative) and
$I_\nu$ denotes the modified Bessel function of the first kind and order $\nu$, 
cf. \cite[Chapter 5]{Leb}; considered on $\mathbb R_+$, it is real positive and smooth for any $\nu > -1$. 
The semigroup property of the kernel $G_t^{\alpha,\widehat L}(x,y)$ is reflected in the identity 
\begin{equation}\label{Ch-K}
\int_0^\infty G_t^{\alpha,\widehat L}(x,z)G_s^{\alpha,\widehat L}(z,y)\,dm_\alpha(z)=
G_{t+s}^{\alpha,\widehat L}(x,y), \qquad x,y\in\mathbb R^d_+,
\end{equation}
which can be independently verified by means of the formula (cf. \cite[Formula 2.15.20 (8)]{PBM})
$$
\int_0^\infty I_\nu(az)I_\nu(bz)\exp(-pz^2)\, zdz=\frac1{2p}\exp\Big(\frac{a^2+b^2}{4p}\Big)
I_{\nu}\Big(\frac{ab}{2p}\Big), \qquad a,b,p>0, \quad\nu>-1.
$$
Since ${\lag}^{\alpha}_{(0,\ldots,0)}\equiv \textrm{const.}$, 
we see from \eqref{spec} that $T_t^{\alpha,{{\lag}}}\textbf{1}=\textbf{1}$; in particular, 
\begin{equation*}
\int_0^\infty G_t^{\alpha,\widehat L}(x,y)\,dm_\alpha(y)=1,\qquad  x\in \mathbb R^d_+.
\end{equation*}
This together with \eqref{Ch-K} shows that the kernel $G_t^{\alpha,\widehat L}(x,y)$ defines
a \textit{Markov semigroup}, see \cite{B}; \eqref{Ch-K} is then called the Chapman-Kolmogorov identity.

The action of $T_t^{\alpha,{{\lag}}}$ can be extended by \eqref{int_pol} to 
$L^1(\R_+,\, m_\alpha)$ (that includes all $L^p(\R_+,\, m_\alpha)$, $1\le p\le\infty$). 
Indeed, the integral in \eqref{int_pol} converges for every $f \in L^1(\R_+,\, m_\alpha)$ and every
$x \in \R_+$, as can be seen by applying the standard asymptotics for $I_{\nu}$,
cf. \cite[(5.16.4),\,(5.16.5)]{Leb},
\begin{equation} \label{bes}
I_\nu(z)\simeq z^\nu,\quad z\to 0^+;\qquad I_\nu(z)\simeq z^{-1/2}e^z, \quad z\to \infty.
\end{equation}
The fact that $T_t^{\alpha,{{\lag}}}\textbf{1}=\textbf{1}$ and the 
positivity of $T_t^{\alpha,{{\lag}}}$ give 
$\|T_t^{\alpha,{{\lag}}}\|_{\infty\to\infty}=1$. Then by the symmetry we also have 
$\|T_t^{\alpha,{{\lag}}}\|_{1\to 1}=1$, and by interpolation we conclude that,
for any $1\le p \le \infty$,
$$
\|T_t^{\alpha,{{\lag}}}\|_{p\to p}\le 1, \qquad t>0.
$$
The same conclusion follows from a simple lemma, which we formulate below for further reference.
This result can be regarded as a special case of the Schur test. In particular, it shows that 
a semigroup $\{T_t\}$ defined on $L^1+L^{\infty}$ by means of a nonnegative symmetric kernel is
$L^p$-contractive if and only if $T_t\textbf{1} \le \textbf{1}$ for all $t$.
\begin{lemma}\label{Schur}
Let $K(x,y)$ be a nonnegative symmetric kernel on $(X,\mu)$ that satisfies 
$$
\int_X K(x,y)\,d\mu(y)\le B, \qquad x\in X.
$$
Then, for each $1\le p\le\infty$, the integral operator $Tf(x)=\int_X K(x,y)f(y)\,d\mu(y)$ 
is bounded on $L^p(X,\mu)$ and $\|T\|_{p\to p}\le B$.
\end{lemma}

Summarizing, for each $\alpha \in (-1,\infty)^d$, the semigroup $\{T_t^{\alpha,{{\lag}}}\}$
is a symmetric diffusion semigroup, which is Markovian. 
Let us mention that $\{T_t^{\alpha,{{\lag}}}\}$ is also \textit{hypercontractive}, which means some
smoothing properties of the semigroup, much more subtle than contractivity; see \cite{KS,K}.

The modified semigroups $\{\widetilde{T}_t^{\alpha,{{\lag}},j}\}$, $j=1,\ldots,d$, are generated
in $L^2(\R_+,\, m_\alpha)$ by proper self-adjoint extensions of the operators
$L_{\alpha}^{{\lag}} + (\alpha_j+1\slash 2+x_j)\slash (2x_j)$, see \cite[Section 2]{No2}.
Their integral representation, valid for $f \in L^1(\R_+,\, m_\alpha)$, is
$$
\widetilde{T}_t^{\alpha,{{\lag}},j}f(x) = e^{-t}\int_{\R_+}\sqrt{x_j y_j}
	G_t^{\alpha+e_j,{\lag}}(x,y) f(y)\, dm_{\alpha}(y), \qquad x \in \R_+, \quad t>0,
$$
where $e_j$ is the $j$th coordinate vector in $\R$. Note that in view of \eqref{bes}
the sharp range of $\alpha$'s for which $\{\widetilde{T}_t^{\alpha,{{\lag}},j}\}$ is well defined on all
$L^p(\R_+,\,m_\alpha)$, $1\le p\le\infty$, and maps $L^p(\R_+,\,m_\alpha)$ into itself, is 
$$
\mathcal{A}_j^{\lag}=
\big\{\alpha=(\alpha_1,\ldots,\alpha_d): \alpha_j>-3/2, \alpha_i>-1\,\, {\rm for}\,\, i\neq j\big\}.
$$ 

It turns out that $\{\widetilde{T}_t^{\alpha,{{\lag}},j}\}$, in contrast with $\{T_t^{\alpha,\lag}\}$, 
is not conservative. It was computed in \cite[p.\,234]{No2} that
\begin{equation} \label{T1polmod}
\widetilde{T}_t^{\alpha,{{\lag}},j}\textbf{1}(x) = e^{-t\slash 2}\;
	\frac{\Gamma(\alpha_j+ 3 \slash 2)}{\Gamma(\alpha_j+2)} 
		\bigg( \frac{x_j e^{-t}}{1-e^{-t}} \bigg)^{1\slash 2}
		{_1F_1}\bigg(\frac{1}2;  \alpha_j+2;  -\frac{x_j e^{-t}}{1-e^{-t}} \bigg), \qquad x \in \R_+,
\end{equation}
where ${_1F_1}$ denotes the confluent hypergeometric function, cf. \cite[Section 9.9]{Leb}.
This was achieved by expanding $\widetilde{T}_t^{\alpha,{{\lag}},j}\textbf{1}$
with respect to a system of `differentiated' Laguerre polynomials, then applying
$\widetilde{T}_t^{\alpha,{{\lag}},j}$ spectrally and summing back the resulting expansion.
Actually, the same can be obtained more directly by combining the formula 
\cite[(9.11.2)]{Leb} (note a misprint there)
\begin{equation} \label{leb9112}
{_1F_1}(a; b; z) = e^{z} {_1F_1}(b-a;b;-z), \qquad b\neq 0,-1,-2,\ldots,
\end{equation}
with the following result.
\begin{lemma}[{\cite[Formula 2.15.5 (4)]{PBM}}] \label{prud}
Given $p,q>0$ and $\beta,\nu \in \mathbb{R}$, $\beta+\nu>0$, $\nu \neq -1,-2,\ldots$, we have
$$ 
 \int_0^{\infty} y^{\beta-1} e^{-py^2} I_{\nu}(qy)\, dy 
  = \frac{q^{\nu}}{2^{\nu+1} p^{(\beta+\nu)\slash 2}} \; \frac{\Gamma(\frac{\beta+\nu}2)}{\Gamma(\nu+1)} \;
	{{_1F_1}}\bigg( \frac{\beta+\nu}2; \nu+1; \frac{q^2}{4p}\bigg). 
$$
\end{lemma}

From \eqref{T1polmod} it can be seen that 
$\|\widetilde{T}_t^{\alpha,{{\lag}},j}\textbf{1}\|_{\infty}>1$ for some $t>0$ 
if and only if $-3\slash 2<\alpha_j<-1\slash 2$; otherwise
$\|\widetilde{T}_t^{\alpha,{{\lag}},j}\textbf{1}\|_{\infty} = e^{-t\slash 2}$.
These facts are essentially pointed out in \cite[p.\,234]{No2}, but without a full justification.
Here we take an opportunity to fill up this gap. Clearly, 
in view of the tensor product structure of the kernels involved, 
it is enough to check, see \eqref{leb9112}, that the function
$u \mapsto [{\Gamma(\alpha_j+ 3 \slash 2)}\slash{\Gamma(\alpha_j+2)}] \sqrt{u} e^{-u}
{_1F_1}(\alpha_j+3\slash 2;  \alpha_j+2; u)$ acting on $\mathbb{R}_+$
has values bigger than $1$ if and only if $\alpha_j < -1\slash 2$,
and when $\alpha_j \ge -1\slash 2$ the closure of its range contains the point $\{1\}$.
This, however, follows readily from the next lemma applied with $\eta=\alpha_j+3\slash 2$ and
$\gamma=\alpha_j+2$. To state the result, it is convenient to
define an auxiliary function, which comes into play when investigating $L^{\infty}$-contraction property
of Laguerre semigroups. Given $\gamma \ge \eta > 0$, let
$$
H_{\eta,\gamma}(u) = \frac{\Gamma(\eta)}{\Gamma(\gamma)} u^{\gamma-\eta} e^{-u}
	{_1F_1}(\eta; \gamma; u), \qquad u >0.
$$
Note that $H_{\eta,\gamma}(u)>0$ for $u > 0$, as can be seen immediately from the 
hypergeometric series defining ${_1F_1}$, see \cite[(9.9.1)]{Leb}.
Moreover, the case $\eta=\gamma$ is trivial in the sense that $H_{\gamma,\gamma}(u) = \textbf{1}$, 
since ${_1F_1}(\gamma; \gamma; u)=e^{u}$.

The following estimates of $H_{\eta,\gamma}$ are crucial.
\begin{lemma} \label{lem:H}
Let $\gamma\ge\eta>0$ be fixed.
\begin{itemize}
\item[(a)]
If either $\eta=\gamma$ or $\eta \ge 1$, then $\|H_{\eta,\gamma}\|_{\infty}=1$
and $\lim_{u \to \infty} H_{\eta,\gamma}(u) = 1$.
\item[(b)]
If $\eta \neq \gamma$ and $\eta<1$, then $1< \|H_{\eta,\gamma}\|_{\infty} < \infty$
and there exists $u_0>0$ such that
$$
H_{\eta,\gamma}(u) > 1, \qquad u\ge u_0.
$$
\end{itemize}
\end{lemma}

\begin{proof}
The case $\eta=\gamma$ is trivial, so assume that $\eta<\gamma$. Using \eqref{leb9112} we write 
$$
H_{\eta,\gamma}(u) = \frac{\Gamma(\eta)}{\Gamma(\gamma)} u^{\gamma-\eta} {_1F_1}(\gamma-\eta; \gamma; -u).
$$
Then the integral representation, see \cite[(9.11.1)]{Leb}, which is valid for $0<a<b$,
$$
{_1F_1}(a;b;z) = \frac{\Gamma(b)}{\Gamma(a)\Gamma(b-a)} \int_0^{1} s^{a-1}(1-s)^{b-a-1}e^{zs}\, ds,
$$
together with the change of variable $s\to s\slash u$, leads to
$$
H_{\eta,\gamma}(u) = \frac{1}{\Gamma(\gamma-\eta)} \int_0^u s^{\gamma-\eta-1} e^{-s}
	\Big(1-\frac{s}{u}\Big)^{\eta-1}\, ds.
$$
We now see that when $\eta \ge 1$,
$$
H_{\eta,\gamma}(u) < \frac{1}{\Gamma(\gamma-\eta)}\int_0^{\infty} s^{\gamma-\eta-1}e^{-s}\, ds = 1, \qquad
	u >0,
$$
and by the monotone convergence theorem, $\lim_{u\to\infty}H_{\eta,\gamma}(u)=1$. This proves (a).

To show (b) we first note that the boundedness of $H_{\eta,\gamma}$  on $(0,\infty)$ is justified by the
continuity of ${_1F_1}(\eta;\gamma;\cdot)$ and the fact that, see \cite[(9.9.1), (9.12.8)]{Leb},
$$
\lim_{u \to 0^+} {_1F_1}(\eta;\gamma;u)=1 \qquad \textrm{and} \qquad
{_1F_1}(\eta;\gamma;u) \simeq \frac{\Gamma(\gamma)}{\Gamma(\eta)}e^{u}u^{\eta-\gamma}, \quad u \to \infty.
$$
Next we observe that given $\varepsilon >0$, there exists $u_0>0$ such that
\begin{equation}\label{ax}
\int_u^{\infty} s^{\gamma-\eta-1}e^{-s}\, ds < \varepsilon \int_{u\slash 2}^{u} s^{\gamma-\eta-1}e^{-s}\,ds,
	\qquad u \ge u_0.
\end{equation}
Indeed, assuming that $u$ is so large that $s^{\gamma-\eta-1}e^{-s}<e^{-3s\slash 4}$ and
$s^{\gamma-\eta-1}e^{-s}>e^{-5s\slash 4}$ for $s\ge u\slash 2$, it is enough to check that
$$
\int_u^{\infty}e^{-3s\slash 4}\, ds < \varepsilon \int_{u\slash 2}^u e^{-5s\slash 4}\, ds
$$
for $u$ sufficiently large; this, however, is immediate.

Assume that $\eta < 1$. Taking $\varepsilon = 2^{1-\eta}-1$ in \eqref{ax}, we get for $u \ge u_0$
\begin{align*}
\int_0^u s^{\gamma-\eta-1} e^{-s} \Big(1-\frac{s}u\Big)^{\eta-1}\, ds & > \int_0^{u\slash 2}
	s^{\gamma-\eta-1}e^{-s}\, ds + \int_{u\slash 2}^{u} s^{\gamma-\eta-1}e^{-s}
	\Big(1-\frac{u\slash 2}u\Big)^{\eta-1}\, ds \\
& = \int_0^{u\slash 2} s^{\gamma-\eta-1}e^{-s}\, ds + 2^{1-\eta}\int_{u\slash 2}^u
 s^{\gamma-\eta-1}e^{-s}\,ds \\
& > \int_0^{\infty} s^{\gamma-\eta-1}e^{-s}\, ds = \Gamma(\gamma-\eta).
\end{align*}
The conclusion follows.
\end{proof}

Properties of $\{\widetilde{T}_t^{\alpha,{{\lag}},j}\}$, $j=1,\ldots,d$, deduced above, the boundedness
of $H_{\alpha_j+\frac{3}2,\alpha_j+2}$ and Lemma \ref{Schur} imply the following.
\begin{propo}
Let $j\in\{1,\ldots,d\}$ be fixed and $\alpha \in \mathcal{A}_j^{\lag}$. 
If $\alpha_j \ge -1\slash 2$, then $\{\widetilde{T}_t^{\alpha,{{\lag}},j}\}$ is 
a submarkovian (but not Markovian) symmetric diffusion semigroup satisfying, for each $1\le p \le \infty$,
$$
 \|\widetilde{T}_t^{\alpha,{{\lag}},j}\|_{p\to p} \le e^{-t\slash 2}, \qquad t>0.
$$
If $\alpha_j<-1\slash 2$, then $\{\widetilde{T}_t^{\alpha,{{\lag}},j}\}$ is not an $L^p$-contractive
semigroup, but there exists a constant $c=c(\alpha)>1$ such that the above estimate holds with
the right-hand side multiplied by $c$.
\end{propo}

\section{Standard Laguerre function semigroup} \label{sec:lag_standard}

The differential operator related to the system $\{\mathcal{L}_k^{\alpha}: k \in \mathbb{N}^d\}$ is
$$
 L_\alpha^{\mathcal L}= - \sum _{i=1}^{d}\left( x_i\frac{\partial^2}{\partial x_i^2}+
  \frac{\partial}{\partial x_i}-\Big(\frac{x_i^2+\alpha _i^2}{4x_i}\Big)\right).
$$
It is formally symmetric and positive in $L^2(\R_{+})$ and we have 
$L_\alpha^{\mathcal L}\mathcal L_k^{\alpha}=(|k|+(|\alpha|+d)/2)\mathcal L_k^{\alpha}$.
The operator $L_\alpha^\mathcal L$ has a natural self-adjoint extension (denoted by the same symbol) 
whose spectral decomposition is given by the $\mathcal L_k^{\alpha}$.
The corresponding heat semigroup $\{T_t^{\alpha,\mathcal L}\}=\{\exp(-tL_\alpha^{\mathcal L})\}$ 
is given in $L^2(\R_+)$ by
\begin{equation*} 
T_t^{\alpha,\mathcal L}f=\sum_{n=0}^{\infty} e^{-t(n+(|\alpha|+d)/2)} \sum_{|k|=n}
\langle f,\mathcal L_k^{\alpha} \rangle\mathcal L_k^{\alpha}, \qquad f\in L^2(\R_+).
\end{equation*}
We have the integral representation
\begin{equation} \label{int_cal}
  T_t^{\alpha,\mathcal L}f(x)=\int_{\R_{+}} G_t^{\alpha,\mathcal L}(x,y)f(y)\,dy, 
  \qquad x\in \R_{+},
\end{equation}
where, for $x,y\in \R_{+}$,
\begin{align*}
G^{\alpha,\mathcal L}_t(x,y)&=\sum_{n=0}^\infty e^{-t(n+(|\alpha|+d)/2)} 
\sum_{|k|=n} \mathcal L_k^\alpha(x)\mathcal L_k^\alpha(y)\\
&=e^{-t(|\alpha|+d)/2}G_{t}^{\alpha,{\lag}}(x,y) e^{-\sum_{i=1}^d(x_i+y_i)\slash 2}
\prod_{i=1}^d (x_i y_i)^{\alpha_i\slash 2}\\
&=(2\sinh (t/2))^{-d}\exp\Big({-\frac{1}{2} \coth(t/2)\sum_{i=1}^d (x_i+y_i)}\Big)
\prod^{d}_{i=1}\, I_{\alpha_i}\left(\frac{ \sqrt{x_i y_i}}{\sinh (t/2)}\right).
\end{align*}
When $\alpha \in [0,\infty)^d$, the action of $T_t^{\alpha,\mathcal L}$ extends to
$L^1(\R_+)+L^{\infty}(\R_+)$ by means of \eqref{int_cal}.
Indeed, for such $\alpha$ the asymptotics \eqref{bes} show that the integral in 
\eqref{int_cal} converges for every $f\in L^p(\R_+)$, $1\le p \le \infty$, and every $x \in \R_+$.
However, this is not the case if at least one of the coordinates of $\alpha$ is less than $0$. 
Let $\tilde{\alpha} = \min_{1\le i \le d} \alpha_i$ and assume that $\tilde{\alpha} < 0$.
Then $\{T_t^{\alpha,\mathcal L}\}$ is well defined by \eqref{int_cal} on $L^p(\R_+)$ 
if and only if $p>-2/\tilde\alpha$. The additional requirement that $T_t^{\alpha,\mathcal L}f\in L^p(\R_+)$
whenever $f\in L^p(\R_+)$ forces another restriction on $p$, namely $p<2/(2+\tilde\alpha)$.
Thus the question of $L^p$-contractivity of $\{T_t^{\alpha,\mathcal{L}}\}$ makes sense only
for $\alpha \in [0,\infty)^d$.

We state the main result of this section.
\begin{theor} \label{th:contr-cal}
Let $\alpha\in[0,\infty)^d$. Then $\{T_t^{\alpha,\mathcal L}\}$ is a  submarkovian 
(but not Markovian) symmetric diffusion semigroup. More precisely, for each $1\le p\le\infty$,
\begin{equation*}
\|T_t^{\alpha,\mathcal L}\|_{p\to p}\le  \big(\cosh(t/2)\big)^{-d}, \qquad t>0.
\end{equation*}
\end{theor}
\begin{proof}
It is enough to consider the one-dimensional case, since then the multi-dimensional result follows
easily by the tensor product structure of the semigroup. Using Lemma \ref{prud} we get
\begin{align*}
T^{\alpha,\mathcal L}_{t}\textbf{1}(x)&=\frac1{\cosh (t/2)}\exp\Big(-\frac{x}{2}\tanh (t/2)\Big)\\
&\quad \times\frac{\Gamma(\frac{\alpha}2+1)}{\Gamma(\alpha+1)}\Big(\frac{x}{\sinh t}\Big)^{\alpha/2}
 \exp\Big(-\frac{x}{\sinh t}\Big){\,}_1F_1\Big(\frac{\alpha}2+1;\alpha+1; \frac{x}{\sinh t}\Big).
\end{align*}
This can be written in terms of the function $H_{\eta,\gamma}$ as
$$
T^{\alpha,\mathcal L}_{t}\textbf{1}(x) = \frac1{\cosh (t/2)}\exp\Big(-\frac{x}{2}\tanh (t/2)\Big)\;
H_{\frac{\alpha}2+1,\alpha+1}\Big(\frac{x}{\sinh t} \Big).
$$
Applying now Lemma \ref{lem:H} (a) we see that 
$\|T^{\alpha,\mathcal L}_{t}\textbf{1}\|_\infty\le (\cosh (t/2))^{-1}$,
and the conclusion follows by Lemma \ref{Schur}.
\end{proof}

We remark that Theorem \ref{th:contr-cal} can be proved in another way, without relying on Lemma 
\ref{lem:H}. Given $z>0$, the function $\nu \mapsto I_{\nu}(z)$ is decreasing for $\nu\ge 0$
(see the proof of \cite[Proposition 2.1]{NS} and references given there) 
and hence $G^{\alpha,\mathcal L}_t(x,y)\le G^{\textbf{0},\mathcal L}_t(x,y)$,
where $\textbf{0}=(0,\ldots,0)$. Thus it is sufficient to prove the bound only for 
$\{T^{\textbf{0},\mathcal L}_{t}\}$, and in that case, merely by using Lemma \ref{prud},
$$
T^{\textbf{0},\mathcal L}_{t} \textbf{1}(x) = \big(\cosh (t/2)\big)^{-d}
	\exp\Big(-\frac{1}{2}\tanh (t/2) \sum_{i=1}^d x_i\Big).
$$

The modified semigroups $\{\widetilde{T}_t^{\alpha,\mathcal L,j}\}$, $j=1,\ldots,d$, are generated in
$L^2(\R_+)$ by suitable self-adjoint extensions of the operators $L_{\alpha+e_j}^{\mathcal L}+1/2$,
see \cite[Section 5]{NS1}. Their integral representation is
$$
\widetilde{T}_t^{\alpha,\mathcal L,j}f(x) = e^{-t/2}\int_{\R_+} G_t^{\alpha+e_j,\mathcal L}(x,y)f(y)\, dy,
\qquad x \in \R_+, \quad t>0.
$$
Thus $\{\widetilde{T}_t^{\alpha,\mathcal L,j}\}$ coincide, up to the factor $e^{-t/2}$, 
with the original Laguerre semigroups $\{T_t^{\alpha+e_j,\mathcal L}\}$. 
Therefore the sharp range of $\alpha$'s for which $\{\widetilde{T}_t^{\alpha,\mathcal L,j}\}$ 
is well defined on all $L^p(\R_+)$, $1\le p\le\infty$, and maps $L^p(\R_+)$ into itself, is 
$$
\mathcal{A}_j^{\mathcal{L}} = 
\big\{\alpha: \alpha_j\ge-1, \alpha_i\ge0\,\, {\rm for}\,\, i\neq j\big\}.
$$ 
Moreover, the above results concerning $\{T_t^{\alpha,\mathcal L}\}$ imply the following.
\begin{propo}
Let $j \in \{1,\ldots,d\}$ be fixed and let $\alpha \in \mathcal{A}_j^{\mathcal{L}}$. 
Then $\{\widetilde{T}_t^{\alpha,\mathcal L,j}\}$ is a submarkovian 
(but not Markovian) symmetric diffusion semigroup satisfying, for $1\le p \le \infty$,
$$
\|\widetilde{T}_t^{\alpha,\mathcal{L},j}\|_{p \to p} \le e^{-t/2} \big(\cosh(t/2)\big)^{-d}, \qquad t >0.
$$
\end{propo}

\section{Laguerre function semigroup of Hermite type} \label{sec:lag_herm}

Recall that the system $\{\varphi_k^{\alpha} : k \in \mathbb{N}^d\}$ is an orthonormal basis in $L^2(\R_+)$.
The associated differential operator, 
$$
 L_\alpha^\varphi=
-\Delta +\|x\|^2 + \sum _{i=1}^{d} \frac{1}{x_i^2} \left(\alpha _i^2 -\frac{1}{4}\right),
$$
is formally symmetric and positive in $L^2(\R_{+})$, and we have 
$L_\alpha^\varphi\varphi _k^{\alpha}=(4|k|+2|\alpha |+2d)\varphi _k^{\alpha}$.
The operator $L_\alpha^\varphi$ has a natural self-adjoint extension (denoted by the same symbol) 
whose spectral decomposition is given by the $\varphi_k^{\alpha}$, see \cite[p.\,402]{NS}.
The corresponding heat semigroup
$\{T_t^{\alpha,\varphi}\} = \{\exp(-tL_{\alpha}^{\varphi})\}$ is given in $L^2(\R_+)$ by
\begin{equation*}
T_t^{\alpha,\varphi}f =\sum_{n=0}^{\infty} e^{-t(4n+2|\alpha|+2d)} \sum_{|k|=n}
\langle f,\varphi _k^{\alpha} \rangle\varphi_k^{\alpha}, \qquad f\in L^2(\R_+).
\end{equation*}
We have the integral representation
\begin{equation} \label{int_phi}
 T_t^{\alpha,\varphi}f(x)=\int_{\R_{+}}G^{\alpha,\varphi}_t(x,y)f(y)\,dy,   \qquad x\in \R_{+},
\end{equation}
where, with the notation $x^2=(x_1^2,\ldots,x_d^2)$, for $x,y\in \R_{+}$,
\begin{align*}
G^{\alpha,\varphi}_t(x,y)&=\sum_{n=0}^\infty e^{-t(4n+2|\alpha|+2d)} 
\sum_{|k|=n} \varphi_k^\alpha(x)\varphi_k^\alpha(y)\\
& = 2^d e^{-2t(|\alpha|+d)}G_{4t}^{\alpha,{\lag}}(x^2,y^2) e^{-(\|x\|^2+\|y\|^2)\slash 2} 
\prod_{i=1}^d (x_i y_i)^{\alpha_i+1\slash 2} \\
&=(\sinh 2t)^{-d}\exp\Big({-\frac{1}{2} \coth(2t)\big(\|x\|^{2}+\|y\|^{2}\big)}\Big)
\prod^{d}_{i=1} \sqrt{x_i y_i}\, I_{\alpha_i}\left(\frac{x_i y_i}{\sinh 2t}\right).
\end{align*}

When $\alpha \in [-1\slash 2,\infty)^d$, the action of $T_t^{\alpha,\varphi}$ can be extended
by \eqref{int_phi} to $L^1(\R_+)+L^{\infty}(\R_+)$. Indeed, as can be easily verified by means of
the asymptotics \eqref{bes}, the integral in \eqref{int_phi} converges for every $f\in L^p(\R_+)$,
$1\le p \le \infty$, and every $x \in \R_+$, provided that $\alpha \in [-1\slash 2,\infty)^d$.
However, the case when at least one of the coordinates of $\alpha$ is less than $-1\slash 2$ is
different. Recall that $\tilde{\alpha} = \min_{1\le i \le d} \alpha_i$ 
and assume that $\tilde{\alpha} < -1\slash 2$.
Then $T_t^{\alpha,\varphi}$ is well defined on $L^p(\R_+)$ by \eqref{int_phi} if and only if
$p> 2\slash (2\tilde{\alpha}+3)$. Moreover, the requirement that $T_t^{\alpha,\varphi}f \in L^p(\R_+)$
whenever $f \in L^p(\R_+)$, imposes further restriction on $p$, namely $p < -2\slash (2\tilde{\alpha}+1)$.
Thus the question of $L^p$-contractivity of $\{T_t^{\alpha,\varphi}\}$ makes sense only for 
$\alpha \in [-1\slash 2,\infty)^d$.

The main result of this section reads as follows.
\begin{theor} \label{th:contr-phi}
Let $\alpha \in [-1\slash 2,\infty)^d$. If $\alpha \in (\{-1\slash 2\}\cup[1\slash 2,\infty))^d$,
then $\{T_t^{\alpha,\varphi}\}$ is a submarkovian (but not Markovian) symmetric diffusion semigroup
satisfying, for $1\le p \le \infty$,
$$
\|T_t^{\alpha,\varphi}\|_{p \to p} \le (\cosh 2t)^{-d\slash 2}, \qquad t>0.
$$
If $\alpha_i \in (-1\slash 2,1\slash 2)$ for some $i=1,\ldots,d$, then $\{T_t^{\alpha,\varphi}\}$
is not an $L^p$-contractive semigroup, but there exists a constant $c=c(\alpha)>1$ such that
the above estimate holds with the right-hand side multiplied by $c$.
\end{theor}
\begin{proof}
We consider first the one-dimensional case. Using Lemma \ref{prud} we compute
\begin{align*}
T^{\alpha,\varphi}_{t}\textbf{1}(x) & = \frac1{\sqrt{\cosh2t}} \exp\Big(-\frac{x^2}2 \tanh 2t \Big)\\
& \quad \times \frac{\Gamma(\frac\alpha2+\frac34)}{\Gamma(\alpha+1)}
\Big(\frac{x^2}{\sinh 4t} \Big)^{\frac\alpha2+\frac14}
\exp\Big(-\frac{x^2}{\sinh 4t}\Big)  {_1F_1}\Big(\frac\alpha2+\frac34;\alpha+1; \frac{x^2}{\sinh 4t}\Big).
\end{align*}
This can be written in terms of the function $H_{\eta,\gamma}$ as
$$
T^{\alpha,\varphi}_{t}\textbf{1}(x) = \frac1{\sqrt{\cosh2t}} \exp\Big(-\frac{x^2}2 \tanh 2t\Big)\;
H_{\frac{\alpha}2+\frac{3}4,\alpha+1}\Big(\frac{x^2}{\sinh 4t} \Big).
$$
Then Lemma \ref{lem:H} shows that $\|T_t^{\alpha,\varphi}\textbf{1}\|_{\infty}\le (\cosh 2t)^{-1\slash 2}$
if $\alpha \in \{-1\slash 2\}\cup [1\slash 2,\infty)$, and that 
$\|T_t^{\alpha,\varphi}\textbf{1}\|_{\infty}> 1$ if $\alpha \in (-1\slash 2, 1\slash 2)$ and 
$t$ is sufficiently small. This, together with Lemma \ref{Schur}, gives the desired conclusion.
The last assertion follows by the boundedness of $H_{\frac{\alpha}2+\frac{3}4,\alpha+1}$.

Passing to the multi-dimensional case, we observe that by the tensor product structure of the semigroup
we have
$$
T_t^{\alpha,\varphi}(\textbf{1}\otimes\ldots\otimes\textbf{1})(x) = 
	T_t^{\alpha_1}\textbf{1}(x_1)\cdot \ldots \cdot T_t^{\alpha_d}\textbf{1}(x_d), \qquad x \in\R_+.
$$
Thus the desired positive results are consequences of the already justified one-dimensional case.

To prove the negative part, we assume for simplicity that $d=2$ and $\alpha=(\alpha_1,\alpha_2)$ with
$\alpha_1 \in (-1\slash 2,1\slash 2)$ and $\alpha_2\notin (-1\slash 2,1\slash 2)$;
in the general case the arguments are analogous.
Taking $x_1=c \sqrt{\sinh 4t}$, where $c$ is a sufficiently large constant, by Lemma \ref{lem:H} (b)
we see that 
$$
H_{\frac{\alpha_1}2+\frac{3}4,\alpha_1+1}\Big(\frac{x_1^2}{\sinh 4t} \Big) = 1+2\varepsilon
$$
for some $\varepsilon>0$ independent of $t$. It follows that 
$T_t^{\alpha_1,\varphi}\textbf{1}(x_1)>1+\varepsilon$ for $t>0$ sufficiently small, and hence
$\|T_t^{\alpha_1,\varphi}\textbf{1}\|_{\infty}>1+\varepsilon$ for such $t$.
On the other hand, taking $x_2=1$ and using Lemma \ref{lem:H} (a) we get
$\|T_t^{\alpha_2,\varphi}\textbf{1}\|_{\infty}>1\slash (1+\varepsilon)$, provided that $t$ is sufficiently
small. Altogether, this shows that 
$$
\|T_t^{\alpha,\varphi}(\textbf{1}\otimes\textbf{1})\|_{\infty} =
\|T_t^{\alpha_1,\varphi}\textbf{1}\|_{\infty} \|T_t^{\alpha_2,\varphi}\textbf{1}\|_{\infty} > 1
$$
for $t$ small enough. The proof is finished.
\end{proof}

The modified semigroups $\{\widetilde{T}_t^{\alpha,\varphi,j}\}$, $j=1,\ldots,d$, are generated in
$L^2(\R_+)$ by suitable self-adjoint extensions of the operators $L_{\alpha+e_j}^{\varphi}+2$,
see \cite[Section 4]{NS}. Their integral representation is
$$
\widetilde{T}_t^{\alpha,\varphi,j}f(x) = e^{-2t}\int_{\R_+} G_t^{\alpha+e_j,\varphi}(x,y)f(y)\, dy,
\qquad x \in \R_+, \quad t>0.
$$
Thus $\{\widetilde{T}_t^{\alpha,\varphi,j}\}$ coincide, up to the factor $e^{-2t}$, with the original
Laguerre semigroups $\{T_t^{\alpha+e_j,\varphi}\}$. 
It follows that the sharp range of $\alpha$'s for which $\{\widetilde{T}_t^{\alpha,\varphi,j}\}$ 
is well defined on all $L^p(\R_+)$, $1\le p\le\infty$, and maps $L^p(\R_+)$ into itself, is 
$$
\mathcal{A}_j^{\varphi} =
\big\{\alpha : \alpha_j\ge-3/2, \alpha_i \ge -1/2\,\, {\rm for}\,\, i\neq j\big\}.
$$ 
Moreover, the above results concerning $\{T_t^{\alpha,\varphi}\}$ imply the following.
\begin{propo}
Let $j \in \{1,\ldots,d\}$ be fixed and let $\alpha \in \mathcal{A}_j^{\varphi}$. 
If $\alpha_j \in \{-3\slash 2\}\cup [-1\slash 2,\infty)$
and $\alpha_i \in \{-1\slash 2\} \cup [1\slash 2,\infty)$, $i\neq j$, then 
$\{\widetilde{T}_t^{\alpha,\varphi,j}\}$ is a submarkovian (but not Markovian)
symmetric diffusion semigroup satisfying, for $1\le p \le \infty$,
$$
\|\widetilde{T}_t^{\alpha,\varphi,j}\|_{p \to p} \le e^{-2t} (\cosh 2t)^{-d\slash 2}, \qquad t >0.
$$
If either $\alpha_j \notin \{-3\slash 2\}\cup [-1\slash 2,\infty)$ or 
$\alpha_i \notin \{-1\slash 2\} \cup [1\slash 2,\infty)$ for some $i\neq j$, then
$\{\widetilde{T}_t^{\alpha,\varphi,j}\}$ is not an $L^p$-contraction semigroup,
but there exists a constant $c=c(\alpha)>1$ such that
the above estimate holds with the right-hand side multiplied by $c$.
\end{propo}

\section{ Laguerre function semigroup of convolution type} \label{sec:lag_conv}

In this case the differential operator in question is
$$
L_{\alpha}^{\ell}=
-\Delta +\|x\|^2 - \sum _{i=1}^{d} \frac{2\alpha_i+1}{x_i} \frac{\partial}{\partial x_i}.
$$
It is formally symmetric and positive in $L^2(\R_{+},\,\mu_\alpha)$ and we have 
$L_\alpha^{\ell}\ell_k^{\alpha}=(4|k|+2|\alpha|+2d)\ell_k^{\alpha}$.
The operator $L_\alpha^\ell$ has a natural self-adjoint extension (denoted by the same symbol) 
whose spectral decomposition is given by the $\ell_k^{\alpha}$.

The integral representation of the heat semigroup $\{T_t^{\alpha,\ell}\}=\{\exp({-tL_\alpha^{\ell}})\}$ is
\begin{equation} \label{int_ell}
 T_t^{\alpha,\ell}  f(x)=\int_{\R_{+}} G_t^{\alpha,\ell}(x,y)f(y)\,d\mu_{\alpha}(y),
  \qquad x\in \R_{+},
\end{equation}
where the heat kernel is
\begin{align*}
G^{\alpha,\ell}_t(x,y)&=\sum_{n=0}^\infty e^{-t(4n+2|\alpha|+2d)}
 \sum_{|k|=n}\ell_k^\alpha(x)\ell_k^\alpha(y)\\
&= 2^d e^{-2t(|\alpha|+d)} e^{-(\|x\|^2+\|y\|^2)\slash 2}	G_{4t}^{\alpha,{\lag}}(x^2,y^2)\\
&=(\sinh 2t)^{-d}\exp\Big({-\frac{1}{2} \coth(2t)\big(\|x\|^{2}+\|y\|^{2}\big)}\Big)
\prod^{d}_{i=1} (x_i y_i)^{-\alpha_i} I_{\alpha_i}\left(\frac{x_i y_i}{\sinh 2t}\right).
\end{align*}
For all $\alpha \in (-1,\infty)^d$, the action of $T_t^{\alpha,\ell}$ may be extended
by \eqref{int_ell} to $L^1(\R_+, \,\mu_\alpha)+L^{\infty}(\R_+, \,\mu_\alpha)$. Indeed,
the asymptotics \eqref{bes} easily imply  the integral in \eqref{int_ell} to be convergent for every 
$f\in L^p(\R_+,\,\mu_\alpha)$, $1\le p \le \infty$, and every $x \in \R_+$.

The following result is essentially contained in \cite[Theorem 2.3]{Ste}.
Here we give an independent proof which is based on Lemma \ref{prud}. Note that for
$\alpha=(-1/2,\ldots,-1/2)$, the conclusions of Theorem \ref{th:contr-ell} and Theorem \ref{th:contr-phi}
coincide, as it should be. 
\begin{theor} \label{th:contr-ell}
Let $\alpha\in(-1,\infty)^d$. Then $\{T_t^{\alpha,\ell}\}$ is a submarkovian (but not Markovian) 
symmetric diffusion semigroup satisfying, for $1\le p\le\infty$,
\begin{equation*} 
\|T_t^{\alpha,\ell}\|_{p\to p}\le  (\cosh 2t) ^{-(|\alpha|+d)},\qquad t>0.
\end{equation*}
\end{theor}
\begin{proof}
In the one-dimensional case Lemma \ref{prud} shows that
\begin{equation*}
T^{\alpha,\ell}_{t}\textbf{1}(x)= \frac1{(\cosh2t)^{\alpha+1}} \exp\Big(-\frac{1}2 x^2 \tanh 2t\Big).
\end{equation*}
Then the desired estimate is a consequence of Lemma \ref{Schur}.

The multi-dimensional result follows by the tensor product structure of the semigroup.
\end{proof}

The modified semigroups $\{\widetilde{T}_t^{\alpha,\ell,j}\}$, $j=1,\ldots,d$, are generated in
$L^2(\R_+,\,\mu_\alpha)$ by suitable self-adjoint extensions of the operators
$L_{\alpha}^{\ell}+\frac{2\alpha_j+1}{x_j^2}+2$, see \cite[Section 4]{NS2}. Their integral representation is
$$
\widetilde{T}_t^{\alpha,\ell,j}f(x) 
	= e^{-2t}\int_{\R_+} x_jy_jG_t^{\alpha+e_j,\ell}(x,y)f(y)\, d\mu_\alpha(y),
	\qquad x \in \R_+, \quad t>0.
$$
Note that the sharp range of $\alpha$'s for which $\{\widetilde{T}_t^{\alpha,\ell,j}\}$ 
is well defined on all $L^p(\R_+,\,\mu_\alpha)$, $1\le p\le\infty$, and maps $L^p(\R_+,\,\mu_\alpha)$ 
into itself, is 
$$
\mathcal{A}_j^{\ell} =
\big\{\alpha : \alpha_j>-3/2, \alpha_i>-1\,\, {\rm for}\,\, i\neq j\big\}.
$$ 

\begin{propo}
Let $j \in \{1,\ldots,d\}$ be fixed and let $\alpha \in \mathcal{A}_j^{\ell}$. 
If $\alpha_j\ge -1\slash 2$, then $\{\widetilde{T}_t^{\alpha,\ell,j}\}$ 
is a submarkovian (but not Markovian) symmetric diffusion semigroup satisfying, for $1\le p \le \infty$,
$$
\|\widetilde{T}_t^{\alpha,\ell,j}\|_{p \to p} \le e^{-2t} (\cosh 2t)^{-(|\alpha|+d)}, \qquad t >0.
$$
If  $\alpha_j \in (-3\slash 2,-1\slash 2)$, then $\{\widetilde{T}_t^{\alpha,\ell,j}\}$ is not an
$L^p$-contraction semigroup, but there exists a constant $c=c(\alpha)>1$ such that the above estimate 
holds with the right-hand side multiplied by $c$.
\end{propo}
\begin{proof}
By the product structure and Lemma \ref{prud} we get
$$
T^{\alpha_j,\ell,j}_{t}\textbf{1}(x) = e^{-2t} (\cosh2t)^{-(|\alpha|+d)} 
	\exp\Big(-\frac{1}2 \|x\|^2 \tanh 2t\Big) \;
	H_{\alpha_j+\frac{3}2,\alpha_j+2}\Big(\frac{x^2_j}{\sinh 4t} \Big).
$$
From here we repeat, with suitable adjustments, the arguments proving Theorem \ref{th:contr-phi}.
\end{proof}

\section{Comments on Bessel semigroups} \label{sec:comm}

The methods developed above for investigating $L^p$-contractivity of Laguerre semigroups are perfectly 
applicable in the context of continuous expansions connected with Bessel operators. 
It is remarkable that the results obtained in the Bessel context are closely related to those
for the Laguerre semigroups. In what follows, for the sake of simplicity, we consider the 
one-dimensional situation. It is straightforward to generalize the results to the multi-dimensional setting.
We also omit the discussion of the associated `modified' semigroups and leave details to interested readers.
For all facts concerning the Bessel setting
that are not properly explained below, the reader may consult \cite{BS} and \cite{BHNV}.

To begin with, consider for $\alpha>-1$ the Bessel differential operator
$$
L_\alpha^{\psi}=-\frac{d^2}{dx^2}+\frac1{x^2}\Big(\alpha^2-\frac14\Big), \qquad x>0,
$$
which is symmetric and positive in $L^2(\mathbb R_+)$. The functions
$$
\psi^{\alpha}_\lambda(x)=\sqrt{x\lambda}J_\alpha(x\lambda), \qquad \lambda>0,
$$
are eigenfunctions of $L_\alpha^{\psi}$,
$L_\alpha^{\psi}\psi^{\alpha}_\lambda=\lambda^2\psi^{\alpha}_\lambda$; here 
$J_\nu$ denotes the Bessel function of the first kind and order $\nu$, cf. \cite[Chapter 5]{Leb}. 
The operator $L_\alpha^{\psi}$ has a natural self-adjoint extension for which the spectral decomposition 
is given via the Hankel transform
$$
\mathcal H_\alpha f(\lambda)=\int_0^\infty \psi^{\alpha}_\lambda(x)f(x)\,dx, \qquad\lambda>0.
$$
Note that for $\alpha=-1\slash 2$ and $\alpha=1\slash 2$ we encounter here the cosine and sine
transforms, respectively.
We keep our convention to denote the extension by using the same symbol $L_\alpha^{\psi}$, 
even though there is an inconvenience: if $0<\alpha<1$, then $L_\alpha^{\psi}$
and $L_{-\alpha}^{\psi}$ coincide as differential operators, but their self-adjoint 
extensions are different; we do hope this will not cause any confusion.

It is well known that $\mathcal H_\alpha \circ \mathcal H_\alpha =\textrm{Id}$, 
and so $\mathcal H_\alpha $ is an isometry on $L^2(\mathbb R_+)$. 
The corresponding heat semigroup $\{T_t^{\alpha,\psi}\}=\{\exp(-tL_\alpha^{\psi})\}$, 
is defined by means of the spectral theorem by
$$
T_t^{\alpha,\psi}f= 
\mathcal H_\alpha\big(e^{-t(\cdot)^2}\mathcal H_\alpha f\big), \qquad f\in L^2(\mathbb R_+),
$$
and it has the integral representation
\begin{equation} \label{int_bess}
 T_t^{\alpha,\psi}  f(x)=\int_0^\infty G_t^{\alpha,\psi}(x,y)f(y)\,dy,
  \qquad x\in \R_{+},
\end{equation}
where the heat kernel is given by
$$
G^{\alpha,\psi}_t(x,y)=\int_{0}^\infty e^{-t\lambda^2} 
	\psi^{\alpha}_\lambda(x)\psi^{\alpha}_\lambda(y)\,d\lambda
=\frac1{2t}\sqrt{xy}\exp\bigg({-\frac{x^2+y^2}{4t}}\bigg) I_\alpha\Big(\frac{xy}{2t}\Big).
$$
Notice that this kernel is symmetric and strictly positive. Note also that for 
$\alpha \ge -1\slash 2$, the action of $T_t^{\alpha,\psi}$ can be extended
by \eqref{int_bess} to $L^1(\mathbb{R}_+)+L^{\infty}(\mathbb{R}_+)$, 
and again this can be easily verified by means of the asymptotics \eqref{bes}. 
If $\alpha<-1/2$, then remarks similar to those preceding the statement of Theorem \ref{th:contr-phi} 
are in order. In particular,  the question of $L^p$-contractivity of $\{T_t^{\alpha,\psi}\}$ 
makes sense only for $\alpha \in [-1\slash 2,\infty)$.

\begin{propo} \label{propo:contr-bess1}
Let $\alpha\in[-1/2,\infty)$. If $\alpha \in \{-1\slash 2\}\cup [1\slash 2,\infty)$, 
then $\{T_t^{\alpha,\psi}\}$ is a submarkovian (Markovian if and only if $\alpha=-1\slash 2$) 
symmetric diffusion semigroup. If $\alpha\in (-1\slash 2,1\slash 2)$, then 
$\{T_t^{\alpha,\psi}\}$ is not an $L^p$-contractive semigroup, but there exists a constant 
$c=c(\alpha)>1$ such that, for all $1\le p \le \infty$, 
$\|T_t^{\alpha,\psi}\|_{p \to p}\le c$, $t>0$.
\end{propo}
\begin{proof}
A direct computation based on Lemma \ref{prud} shows that
$$
T_t^{\alpha,\psi}\textbf{1}(x)=H_{\frac{\alpha}2+\frac{3}4,\alpha+1}\Big(\frac{x^2}{4t} \Big).
$$
The conclusion then follows by Lemmas \ref{lem:H} and \ref{Schur}.
The identity $T_t^{\alpha,\psi}\textbf{1}=\textbf{1}$ holds if and only if $\alpha=-1\slash 2$;
see the proof of Lemma \ref{lem:H} for the relevant property of the function $H_{\eta,\gamma}$.
\end{proof}

Another Bessel operator considered in the literature is 
$$
L_\alpha^{\Psi}=-\frac{d^2}{dx^2}-\frac{2\alpha+1}{x}\frac d{dx}, \qquad x>0,
$$
which is symmetric and positive in $L^2(\mathbb R_+,\,\mu_\alpha)$, $\alpha>-1$. The functions
$$
\Psi^{\alpha}_\lambda(x)=(x\lambda)^{-\alpha}{J_\alpha(x\lambda)}, \qquad \lambda>0,
$$
are eigenfunctions of $L_\alpha^{\Psi}$, 
$L_\alpha^{\Psi}\Psi^{\alpha}_\lambda=\lambda^2\Psi^{\alpha}_\lambda$. 
The operator $L^{\Psi}_{\alpha}$ has a natural self-adjoint extension
(denoted by the same symbol) for which the spectral 
decomposition is given via the modified Hankel transform
$$
  h_\alpha f(\lambda)=\int_0^\infty \Psi^{\alpha}_\lambda(x)f(x)\,d\mu_\alpha(x), \qquad\lambda>0.
$$
Again, it is known that $h_\alpha \circ h_\alpha =\textrm{Id}$, 
and $h_\alpha$ is an isometry on $L^2(\mathbb R_+,\,\mu_\alpha)$. The corresponding heat semigroup
$\{T_t^{\alpha,\Psi}\}=\{\exp(-t L_\alpha^{\Psi})\}$, is defined by means of the spectral theorem by
$$
T_t^{\alpha,\Psi}f = h_\alpha\big(e^{-t(\cdot)^2}  h_\alpha f\big), 
	\qquad f\in L^2(\mathbb R_+,\,\mu_\alpha),
$$
and has the integral representation
\begin{equation} \label{int_bess2}
 T_t^{\alpha,\Psi}  f(x)=\int_0^\infty G_t^{\alpha,\Psi}(x,y)f(y)\,d\mu_{\alpha}(y),
  \qquad x\in \mathbb{R}_{+},
\end{equation}
where the heat kernel is
$$
G^{\alpha,\Psi}_t(x,y)= (xy)^{-(\alpha+1/2)}G^{\alpha, \psi}_t(x,y)
=\frac1{2t}(xy)^{-\alpha}\exp\bigg({-\frac{x^2+y^2}{4t}}\bigg) I_\alpha\Big(\frac{xy}{2t}\Big).
$$
When $\alpha \in (-1,\infty)$, the action of $\{T_t^{\alpha,\Psi}\}$ may be extended
by \eqref{int_bess2} to $L^1(\R_+, \,\mu_\alpha)+L^{\infty}(\R_+, \,\mu_\alpha)$.

The result below for $\alpha \ge -1\slash 2$ is well known in the literature, since
for such $\alpha$ there exists a natural convolution structure for which the Bessel semigroup
is a convolution semigroup, and the convolution kernel is easily seen to have the relevant properties.
\begin{propo} \label{propo:contr-bess2}
Let $\alpha>-1$. Then $\{T_t^{\alpha,\Psi}\}$ is a Markovian symmetric diffusion semigroup. 
\end{propo}
\begin{proof}
This time a direct computation based on Lemma \ref{prud} gives
$T_t^{\alpha,\Psi}\textbf{1}=\textbf{1}$.
Again the conclusion follows by using Lemmas \ref{lem:H} and \ref{Schur}.
\end{proof}

As the reader probably noticed, there is a striking coincidence between the results stated above for 
the two Bessel semigroups and those for the Laguerre semigroups of Hermite and convolution types. 
Heuristically, this may be explained by a similarity of the differential operators: 
$L_\alpha^{\psi}$ and $L_\alpha^{\Psi}$ differ from $L_\alpha^{\varphi}$ and 
$L_\alpha^\ell$, respectively, only by the potential $x^2$ 
(recall that here we restrict ourselves to $d=1$). It is not surprising then, that
the heat kernels $G_t^{\alpha,\psi}$ and $G_t^{\alpha,\Psi}$ have the same 
shapes as $G_t^{\alpha,\varphi}$ and $G_t^{\alpha,\ell}$, respectively. In particular, the
behavior for $x$ and $y$ small is the same. 
This explains why the statements of Propositions \ref{propo:contr-bess1}
and \ref{propo:contr-bess2} are similar to those of Theorems \ref{th:contr-phi} and~\ref{th:contr-ell}.

Finally, it is worth of pointing out that one of the examples we discussed disproves
the following (fallacious) belief: if $L$ is a positive and formally symmetric in $L^2(X,\,\mu)$, 
$X\subset \mathbb R^d$, linear second order differential operator free of zero order term, and 
$\mathbb{L}$ is a self-adjoint extension of it, then $\exp(-t\mathbb{L})\textbf{1}=\textbf{1}$ 
(here we assume that  $\exp(-t\mathbb{L})$ extends onto $L^\infty$). Indeed, $-\frac{d^2}{dx^2}$ 
considered as $L^\psi_{1/2}$ is the example, since for the self-adjoint extension given spectrally 
by the sine transform, we have $T^{1/2,\psi}_t\textbf{1}=H_{1,3/2}((\cdot)^2/4t)\neq \textbf{1}$. 
Notice, however, that for the other self-adjoint extension of $-\frac{d^2}{dx^2}$ given spectrally 
by the cosine transform (when we consider $-\frac{d^2}{dx^2}$  as $L^\psi_{-1/2}$), 
we do have  $T^{-1/2,\psi}_t\textbf{1}=H_{1/2,1/2}((\cdot)^2/4t) = \textbf{1}$.


\end{document}